\documentclass[graybox]{svmult}
\usepackage{type1cm}
\usepackage{makeidx}
\usepackage{graphicx}
\usepackage{multicol}
\usepackage[bottom]{footmisc}
\usepackage{newtxtext}
\usepackage[varvw]{newtxmath}
\makeindex

\def\Rho{\mathrm{P}}
\DeclareSymbolFont{script}{U}{eus}{m}{n}
\DeclareMathSymbol{\Wedge}{0}{script}{"5E}

\begin{document}

\title*{Conformal Loxodromes}
\author{Michael Eastwood}
\institute{Michael Eastwood \at 
University of Adelaide, South Australia 5005, \email{meastwoo@gmail.com}}

\maketitle

\abstract{In conformal differential geometry, there are some distinguished
curves, often known as `conformal circles,' since, on the round sphere, they
are the round circles (and these are conformally invariant).  But on the
two-sphere, the curves of constant compass bearing are also conformally
invariant.  These `loxodromes' admit\linebreak a curved analogue in the realm
of M\"obius geometry.  In this article, these curved analogues are explained
and the fifth order invariant ODE that they satisfy is derived.}

\begin{acknowledgement}
This article is based on a talk given at the conference `Geometry, Analysis,
and Representation Theory of Lie Groups' held at the University of Tokyo in
September 2022, honouring Toshiyuki Kobayashi on the occasion of his sixtieth
birthday.  I would like to thank the organisers of this conference for the
invitation to speak and also for the excellent job they did, especially under
the severe COVID restrictions that were in place at the time.  I would also
like to thank Toshi for many inspirational conversations over the past thirty
years.

\end{acknowledgement}

\section{Introduction} This article is concerned with {\em conformal
geometry\/} in dimension $n\geq 3$ and {\em M\"obius geometry\/} in
dimension~$2$.  These are {\em parabolic geometries\/}~\cite{CS} and, more
precisely, this article is concerned with the {\em distinguished curves\/} of a
parabolic geometry in the particular cases of conformal and M\"obius geometry.
In conformal geometry, one particular type of such distinguished curves is due
to Yano~\cite{Y}, and these are nowadays well-known as {\em conformal
geodesics\/}~\cite{T} or {\em conformal circles\/}~\cite{BE}.  Indeed, they are
so well-known that one might be tempted to suppose that there are no other
options for conformally distinguished curves.  I would like to thank Omid
Makhmali for asking a question concerning distinguished curves in conformal and
CR geometry, which I started to answer by asserting that Yano's conformal
circles are the only option.  But this is not true and the main purpose of this
article is to explain another possibility, namely the {\em conformal
loxodromes\/} of the title.  In two dimensions, i.e.~for M\"obius geometry,
these extra curves, together with conformal circles, cover all possibilities.
Conformal loxodromes also provide a nice complementary notion to that of
conformal circles: whereas it has recently been shown by Cameron, Dunajski, and
Tod~\cite{CDT} that conformal circles cannot {\em spiral\/} (i.e.~they cannot
enter and remain in every neighbourhood of a point without passing through the
point itself), conformal loxodromes on the sphere always spiral.

The final version of this article has benefitted greatly from feedback provided
by Maciej Dunajski, Josef \v{S}ilhan, Lenka Zalabov\'a, and Vojt\v{e}ch
\v{Z}\'adn\'{\i}k. I would also like to thank an anonymous referee for many
valuable suggestions.

\section{Conformal circles}\label{cc} In this section, we shall review a
particular construction of conformal circles.  There are, in fact, many
constructions~\cite{BE,CS,DK,FS,T,Y} but the route we describe here is due to
Doubrov and \v{Z}\'adn\'{\i}k~\cite{DZ} in general parabolic geometry and
implemented in~\cite{EZ} for the case of conformal geometry, ending up with the
equations of Tod~\cite{T} for {\em unparameterised\/} conformal circles.  Here
one may encounter a red herring since conformal circles are naturally equipped
with preferred parameterisations defined up to {\em projective freedom\/} (as
in~\cite{CS,FS}, for example).  Indeed, {\em any\/} curve in a conformal
manifold is equipped with such preferred parameterisations, just as any curve
in a Riemannian manifold, geodesic or not, may be equipped with preferred {\em
arc-length\/} parameterisations.  In~\cite{BE}, the equations of conformal
circles are broken into two sets, one of which controls where they go whilst
the other controls the parameterisations.

Our viewpoint in this article follows Tod~\cite{T}, who writes the equations in
terms of a chosen {\em background metric\/}, as follows.  Let $g_{ab}$ denote
the background metric on our smooth manifold $M$ and $\nabla_a$ the associated
Levi-Civita connection.  For $\gamma\hookrightarrow M$ a smooth oriented 
curve, let $U^a$ denote the unit tangent vector along $\gamma$ in the 
direction of the orientation and let $\partial\equiv U^a\nabla_a$ denote the 
{\em directional derivative\/} along~$\gamma$. This differential operator is 
well-defined on any tensor field defined along $\gamma$ and, in particular, we 
may define the {\em acceleration\/} field along $\gamma$ by
$$A_a\equiv\partial U_a,\quad\mbox{where}\enskip U_a\equiv g_{ab}U^b.$$
As usual, the curve $\gamma$ is a {\em geodesic\/} if and only if the
acceleration~$A_a$ vanishes.  This notion is not conformally invariant. 
Specifically, if $g_{ab}$ is replaced by a {\em conformally equivalent\/} 
metric $\widehat g_{ab}=\Omega^2g_{ab}$, for some smooth function~$\Omega>0$, 
then 
\begin{equation}\label{A-transform}
\widehat A_a=A_a-\Upsilon_a+U^b\Upsilon_bU_a,\quad\mbox{where}\enskip
\Upsilon_a\equiv\Omega^{-1}\nabla_a\Omega.\end{equation}
But, following Tod~\cite{T}, if we introduce the {\em normalised jerk\/}
\begin{equation}\label{jerk}
J_a\equiv\partial A_a+(A^bA_b+\Rho_{bc}U^bU^c)U_a-\Rho_{ab}U^b,\end{equation}
where $\Rho_{ab}$ is the {\em Schouten tensor\/}
\begin{equation}\label{schouten}
\textstyle\Rho_{ab}\equiv\frac1{n-2}(R_{ab}-\frac1{2(n-1)}Rg_{ab}),
\end{equation}
then it turns out that
\begin{equation}\label{jerk-transform}\widehat J_a=\Omega^{-1}J_a
\end{equation}
and so 
\begin{equation}\label{cc-equation}
\partial A_a=\Rho_{ab}U^b-(A^bA_b+\Rho_{bc}U^bU^c)U_a\end{equation}
is a conformally invariant third order ODE, which defines unparameterised 
curves known as {\em conformal circles\/}.

Notice that in this exposition we have implicitly supposed that $n\geq3$, in 
order to define the Schouten tensor~(\ref{schouten}), whose key feature is 
that 
\begin{equation}\label{Rho-transform}\textstyle\widehat\Rho_{ab}
=\Rho_{ab}-\nabla_a\Upsilon_b+\Upsilon_a\Upsilon_b
-\frac12\Upsilon^c\Upsilon_cg_{ab}.\end{equation} In two dimensions, we may
simply assert that $\Rho_{ab}$ is to be part of the structure, associated with
any choice of metric and, if $\widehat g_{ab}=\Omega^2g_{ab}$, then
this {\em Rho-tensor\/} is decreed to transform according
to~(\ref{Rho-transform}).  This is a {\em M\"obius structure\/} in the sense of
Calderbank~\cite{C} and conformal circles are defined for such structures by
exactly the same equation~(\ref{cc-equation}).

It takes some calculation to verify (\ref{jerk-transform}) and it is preferable
to set things up so that this remarkable transformation law is automatic.  In
\cite{EZ}, this was accomplished by unpacking the {\em Cartan connection\/}
formulation of unparameterised distinguished curves in parabolic geometry, due
to Doubrov-\v{Z}\'adn\'{\i}k~\cite{DZ}.  Even so, this was done indirectly (by
using the {\em standard tractor connection\/}).  Here, we take the opportunity
to do so more directly, by using the {\em adjoint tractor connection\/},
defined as follows.

Firstly, we need to define the {\em adjoint tractor bundle\/} and, for this, we
shall need the notion of {\em conformally weighted\/} densities and tensors.
In the presence of a metric~$g_{ab}$, a conformal density $\sigma$ of weight
$w$ is simply a smooth function $\sigma_g$ but if $g_{ab}$ is replaced by a
{\em conformally equivalent\/} metric~$\widehat g_{ab}=\Omega^2g_{ab}$, then we
are obliged to replace the corresponding smooth function by
$\sigma_{\widehat{g}}=\Omega^w\sigma_g$.  We may view conformal densities of
weight~$w$ as sections of a smooth line bundle~$\Wedge^0[w]$ and observe that,
if $M$ is oriented, then this notion is tied to the bundle of $n$-forms via a
canonical identification $\Wedge^n=\Wedge^0[-n]$ (see \cite{BEG} for details).
Conformally weighted tensors are defined similarly: for example, the bundle of
conformally weighted $1$-forms of weight~$2$ is
$\Wedge^1[2]\equiv\Wedge^1\otimes\Wedge^0[2]$, the conformal metric is,
tautologically, a section of the bundle $\bigodot^2\!\Wedge^1[2]$, and raising
and lowering indices using the conformal metric gives a tautological
identification $TM=\Wedge^1[2]$ (again, see~\cite{BEG} for details).  Notice 
that we may better regard the normalised jerk (\ref{jerk}) as a $1$-form of 
weight $-1$ and then (\ref{jerk-transform}) simply says that $J_a$ is 
conformally invariant.

In the presence of a metric~$g_{ab}$, the adjoint tractor bundle
${\mathcal{A}}$ is defined to be a direct sum of four individual tensor
bundles, namely
$${\mathcal{A}}\enskip\equiv\enskip
\begin{array}{c}{}\enskip TM\\ 
\Wedge^2[2]\quad\Wedge^0\\
{}\enskip\Wedge^1\end{array}\enskip=\enskip
\begin{array}{c}{}\enskip\Wedge^1[2]\\ 
\Wedge^2[2]\quad\Wedge^0\\
{}\enskip\Wedge^1\end{array}$$
but if we conformally rescale the metric $\widehat g_{ab}=\Omega^2g_{ab}$, 
then sections of ${\mathcal{A}}$ are obliged to transform as
\begin{equation}\label{adjoint_transformation}
\raisebox{12pt}{\makebox[0pt][l]{\LARGE$\enskip\widehat{\enskip\quad}$}}
\left[\begin{array}{c}\sigma_b\\
\mu_{bc}\quad\nu\\
\rho_b\end{array}\right]
=\left[\begin{array}{c}\sigma_b\\
\mu_{bc}+2\Upsilon_{[b}\sigma_{c]}\quad\nu+\Upsilon^b\sigma_b\\
\rho_b+\Upsilon^a\mu_{ab}-\Upsilon_b\nu-\Upsilon^a\Upsilon_b\sigma_a
+\frac12\Upsilon^a\Upsilon_a\sigma_b\end{array}\right],\end{equation}
where $\Upsilon_a\equiv\Omega^{-1}\nabla_a\Omega$.  It is easy to check that
the bundle ${\mathcal{A}}$ is well-defined.  Notice from
(\ref{adjoint_transformation}) that ${\mathcal{A}}$ is equipped with a
canonical surjection ${\mathcal{A}}\to TM$ (indeed, we may regard
${\mathcal{A}}$ as a canonically defined subbundle of the jet-bundle $J^2TM$).
This is a feature common to all adjoint tractor bundles in parabolic geometry
(see~\cite{CS} for details).  In generality, adjoint tractor bundles come
equipped with invariantly defined connections and, as one might expect, in
conformal geometry the adjoint tractor connection
$\nabla_a:{\mathcal{A}}\to\Wedge^1\otimes{\mathcal{A}}$ may be written in terms
of a chosen metric in the conformal class and its associated Levi-Civita
connection.  The formula is
$$\nabla_a\left[\begin{array}{c}\sigma_b\\
\mu_{bc}\quad\nu\\
\rho_b\end{array}\right]=
\left[\begin{array}{c}\nabla_a\sigma_b-\mu_{ab}-\nu g_{ab}\\
\nabla_a\mu_{bc}-2g_{a[b}\rho_{c]}+2\Rho_{a[b}\sigma_{c]}
\quad\nabla_a\nu+\rho_a+\Rho_a{}^b\sigma_b\\
\nabla_a\rho_b-\Rho_a{}^c\mu_{bc}-\Rho_{ab}\nu\end{array}\right].$$
In any case, from (\ref{Rho-transform}) and (\ref{adjoint_transformation}) it
is straightforward to verify that this formula gives an invariantly defined
connection.  In particular, for any given curve $\gamma\hookrightarrow M$ we
may use the adjoint tractor connection to define an invariant directional
derivative
$$\partial\equiv U^a\nabla_a
:{\mathcal{A}}|\gamma\to{\mathcal{A}}[-1]|\gamma,$$
the conformal weight arising because if we replace the metric $g_{ab}$ by
$\widehat g_{ab}=\Omega^2g_{ab}$, then we must replace $U^a$ by 
$\widehat U^a=\Omega^{-1}U^a$.  For a chosen metric $g_{ab}$ in the conformal 
class, let us record the formula for this invariantly defined directional 
derivative:
$$\partial\left[\begin{array}{c}\sigma_b\\
\mu_{bc}\quad\nu\\
\rho_b\end{array}\right]=
\left[\begin{array}{c}\partial\sigma_b-U^a\mu_{ab}-\nu U_b\\
\partial\mu_{bc}-2U_{[b}\rho_{c]}+2U^a\Rho_{a[b}\sigma_{c]}
\quad\partial\nu+U^a\rho_a+U^a\Rho_a{}^b\sigma_b\\
\partial\rho_b-U^a\Rho_a{}^c\mu_{bc}-U^a\Rho_{ab}\nu\end{array}\right].$$

With its index lowered, the unit velocity $U_a$ is a section of $\Wedge^1[1]$. 
Thus, it has the wrong weight in order that 
$$\left[\begin{array}{c}0\\
0\quad 0\\
U_b\end{array}\right]$$
be an invariantly defined section of~${\mathcal{A}}$. The best one can say is 
that arbitrary multiples thereof invariantly define a line subbundle 
of~${\mathcal{A}}|_\gamma$. Applying $\partial$ gives
$$\partial\left[\begin{array}{c}0\\
0\quad 0\\
U_b\end{array}\right]=
\left[\begin{array}{c}0\\
0\quad 1\\
A_b\end{array}\right]$$
and we conclude that 
$${\mathrm{span}}\left(\left[\begin{array}{c}0\\
0\quad 0\\
U_b\end{array}\right],
\left[\begin{array}{c}0\\
0\quad 1\\
A_b\end{array}\right]\right)$$
invariantly defines a rank~$2$ subbundle of ${\mathcal{A}}|_\gamma$.
Alternatively, the transformation law (\ref{adjoint_transformation}) may be
employed to check this directly:
\begin{equation}\label{direct_check}
\raisebox{12pt}{\makebox[0pt][l]{\LARGE$\enskip\widehat{\quad}$}}
\left[\begin{array}{c}0\\
0\quad 0\\
U_b\end{array}\right]
=\left[\begin{array}{c}0\\
0\quad 0\\
U_b\end{array}\right]\quad\mbox{and}\quad
\raisebox{12pt}{\makebox[0pt][l]{\LARGE$\enskip\widehat{\quad}$}}
\left[\begin{array}{c}0\\
0\quad 1\\
A_b\end{array}\right]
=\left[\begin{array}{c}0\\
0\quad 1\\
\widehat A_b\end{array}\right]
-\Upsilon^aU_a\left[\begin{array}{c}0\\
0\quad 0\\
U_b\end{array}\right].\end{equation}
In fact, there is an invariantly defined rank $3$ subbundle:
\begin{proposition}\label{the_bundle_B} 
For any smooth oriented curve $\gamma\hookrightarrow M$,
$${\mathcal{B}}\equiv{\mathrm{span}}\left(\left[\begin{array}{c}0\\
0\quad 0\\
U_b\end{array}\right],
\left[\begin{array}{c}0\\
0\quad 1\\
A_b\end{array}\right],
\left[\begin{array}{c}U_b\\
2U_{[b}A_{c]}\quad 0\\
0\end{array}\right]\right)$$
defines an invariant rank~$3$ subbundle of~${\mathcal{A}}|_\gamma$. 
\end{proposition}
\begin{proof} A calculation yields
$$\begin{array}{rcl}
\raisebox{12pt}{\makebox[0pt][l]{\LARGE$\enskip\widehat{\enskip\qquad}$}}
\left[\begin{array}{c}U_b\\
2U_{[b}A_{c]}\quad 0\\
0\end{array}\right]&=&\textstyle
\left[\begin{array}{c}U_b\\
2U_{[b}\widehat A_{c]}\quad 0\\
0\end{array}\right]
+\Upsilon^aU_a\left[\begin{array}{c}0\\
0\quad 1\\
\widehat A_b\end{array}\right]\\
&&{}-(\Upsilon^aA_a+(\Upsilon^aU_a)^2-\frac12\Upsilon^a\Upsilon_a)
\left[\begin{array}{c}0\\
0\quad 0\\
U_b\end{array}\right]\end{array}$$
and, together with~(\ref{direct_check}), the result follows.
\end{proof}
\begin{theorem}\label{mini_thm}
The subbundle ${\mathcal{B}}\subset{\mathcal{A}}|_\gamma$ is preserved by
$\partial$ if and only if $J_a\equiv 0$.
\end{theorem}
\begin{proof}
It remains to calculate
\begin{equation}\label{partial_squared_Phi}\partial\left[\begin{array}{c}0\\
0\quad 1\\
A_b\end{array}\right]
=\left[\begin{array}{c}0\\
0\quad 0\\
J_b\end{array}\right]
-\left[\begin{array}{c}U_b\\
2U_{[b}A_{c]}\quad 0\\
0\end{array}\right]
-(A^cA_c+U^aU^c\Rho_{ac})\left[\begin{array}{c}0\\
0 \quad 0\\
U_b
\end{array}\right]\end{equation}
and
\begin{equation}\label{second_equation_from_the_proof}
\partial\left[\begin{array}{c}U_b\\
2U_{[b}A_{c]}\quad 0\\
0\end{array}\right]
=\left[\begin{array}{c}0\\
2U_{[b}J_{c]}\quad 0\\
0\end{array}\right]
+U^aU^c\Rho_{ac}\left[\begin{array}{c}0\\
0\quad 1\\
A_b\end{array}\right]
-U^aA^c\Rho_{ac}\left[\begin{array}{c}0\\
0\quad 0\\
U_b\end{array}\right]\end{equation}
either of which forces $J_b=\zeta U_b$ for some scalar~$\zeta$. However, 
from~(\ref{jerk}), 
$$U^aJ_a=U^a\partial A_a+A^aA_a=\partial(U^aA_a)=\partial(0)=0$$
so $\zeta=0$ and we are done.
\end{proof}
{\bf Remark}\enskip As pointed out an anonymous referee, one can effectively
replace the bundle ${\mathcal{B}}$ in Proposition~\ref{the_bundle_B} by
$${\mathrm{span}}\left(\Phi,\partial\Phi,\partial^2\Phi\right),
\quad\mbox{where}\enskip
\Phi\equiv\left[\begin{array}{c}0\\
0\quad 0\\
U_b\end{array}\right].$$
This bundle has the advantage of being manifestly conformally invariant and, 
when~$J_a\equiv 0$, it is clear from (\ref{partial_squared_Phi}) that it 
coincides with~${\mathcal{B}}$. Furthermore, together with the observation 
that $U^aJ_a=0$, we may deduce from
(\ref{partial_squared_Phi}) and (\ref{second_equation_from_the_proof}) that
$$\partial^2\Phi=\left[\begin{array}{c}0\\
0\quad 0\\
J_b\end{array}\right]
-\left[\begin{array}{c}U_b\\
2U_{[b}A_{c]}\quad 0\\
0\end{array}\right]
\bmod\Phi$$
and
$$\partial^3\Phi=\left[\begin{array}{c}0\\
0\quad 0\\
\partial J_b\end{array}\right]
-\left[\begin{array}{c}0\\
4U_{[b}J_{c]}\quad 0\\
0\end{array}\right]
\bmod\Phi,\partial\Phi.$$
Therefore, ${\mathrm{span}}\left(\Phi,\partial\Phi,\partial^2\Phi\right)
\subset{\mathcal{A}}|_\gamma$ is preserved by $\partial$ if and only if
$J_a\equiv 0$, as stated by the referee, who also points out that this
conclusion resembles the corresponding characterisation of unparameterised
conformal circles given in \cite[Proposition~3.3]{SZ}, namely that a certain 
rank 3 subbundle $\langle T,T^\prime,T^{\prime\prime}\rangle$ of the {\em 
standard\/} tractor bundle is parallel along~$\gamma$. The advantage 
of~${\mathcal{B}}$, however, is that it depends one fewer jets of $\gamma$ 
than does ${\mathrm{span}}\left(\Phi,\partial\Phi,\partial^2\Phi\right)$ and it
is a subbundle of the symmetry bundle derived from circles in the flat model, 
as discussed in the remainder of this section. For more on the 
characterisation of conformal circles via standard tractors, see~\cite{GST}.

\medskip 
Theorem~\ref{mini_thm} shows that the condition $J_b\equiv 0$ is conformally 
invariant. To relate this characterisation to the general criterion 
of~\cite{DZ}, we need to enlarge the subbundle~${\mathcal{B}}$ by means of 
a complementary subbundle:
$${\mathcal{C}}\equiv\left\{\left[\begin{array}{c}0\\
F_{bc}\quad 0\\
F_{bc}A^c\end{array}\right]\enskip\mbox{s.t.}\enskip
U^bF_{bc}=0\right\}.$$
Since
$$\raisebox{12pt}{\makebox[0pt][l]{\LARGE$\enskip\widehat{\enskip\quad}$}}
\left[\begin{array}{c}0\\
F_{bc}\quad 0\\
F_{bc}A^c\end{array}\right]
=\left[\begin{array}{c}0\\
F_{bc}\quad 0\\
F_{bc}A^c-F_{bc}\Upsilon^c\end{array}\right]
=\left[\begin{array}{c}0\\
F_{bc}\quad 0\\
F_{bc}\widehat A^c\end{array}\right],$$
we see that ${\mathcal{C}}\subset{\mathcal{A}}|_\gamma$ is invariantly defined
along $\gamma\hookrightarrow M$ and since
$$\partial\left[\begin{array}{c}0\\
F_{bc}\quad 0\\
F_{bc}A^c\end{array}\right]
=\left[\begin{array}{c}0\\
0\quad 0\\
F_{bc}J^c\end{array}\right]
+\left[\begin{array}{c}0\\
\tilde F_{bc}\quad 0\\
\tilde F_{bc}A^c\end{array}\right],$$
where $\tilde F_{bc}\equiv \partial F_{bc}-2U_{[b}F_{c]d}A^d$, which
satisfies~$U^b\tilde F_{bc}=0$, we see that ${\mathcal{C}}$ is preserved by 
$\partial$ along $\gamma$ when $J_a\equiv 0$. Theorem~\ref{mini_thm} can 
thus be modified as follows.
\begin{theorem}\label{full_thm}
The subbundle ${\mathcal{S}}\equiv{\mathcal{B}}\oplus{\mathcal{C}}
\subset{\mathcal{A}}_\gamma$ is preserved by $\partial$ if and only
if~$J_a\equiv 0$.
\end{theorem}
As detailed in~\cite{EZ}, the advantage of Theorem~\ref{full_thm} is that the
subbundle ${\mathcal{S}}$ has a geometric interpretation, roughly as follows.
Firstly, as its name suggests, the adjoint tractor bundle is modelled on the
adjoint representation ${\mathfrak{so}}(n+1,1)$ of
${\mathrm{SO}}^\uparrow(n+1,1)$, more precisely on the homogeneous vector
bundle
$$\begin{array}{c}
{\mathrm{SO}}^\uparrow(n+1,1)\times_P{\mathfrak{so}}(n+1,1)\\
\downarrow\\
{\mathrm{SO}}^\uparrow(n+1,1)/P=S^n,
\end{array}$$
where ${\mathrm{SO}}^\uparrow(n+1,1)$ is the connected component of the
identity of ${\mathrm{SO}}(n+1,1)$ acting on $S^n$ by conformal transformations
with $P$ as stabiliser subgroup.  The adjoint tractor connection is modelled on
the flat connection on this homogeneous bundle induced by the Maurer-Cartan
form (see, e.g.~\cite{CS}).  The Lie algebra ${\mathfrak{so}}(n+1,1)$ may be realised as the space of
conformal Killing fields on the round $n$-sphere $S^n$ and, as such, comes
equipped with a family of Lie subalgebras, namely the symmetry algebras of the
family of round circles $S^1\hookrightarrow S^n$.  As detailed in~\cite{EZ},
these symmetry algebras have dimension
$$3+\frac{(n-1)(n-2)}2=\dim{\mathcal{B}}+\dim{\mathcal{C}}$$
and, indeed, the subbundle ${\mathcal{S}}\subset{\mathcal{A}}|\gamma$ is
modelled on this family of subalgebras.  See~\cite{EZ} for a precise unpacking
of the formulation of unparameterised distinguished curves in parabolic
geometry due to Doubrov-\v{Z}\'adn\'{\i}k~\cite{DZ} in terms of the adjoint
tractor bundle, arriving at Theorem~\ref{full_thm} above for conformal
geometry.  

In $2$ dimensions, the formulation is exactly the same, if not simpler, since
here we may identify $S^2$ as the Riemann sphere ${\mathbb{CP}}_1$, with
${\mathrm{SL}}(2,{\mathbb{C}})$ acting by M\"obius transformations.  In $2$
dimensions ${\mathcal{C}}=0$ so ${\mathcal{S}}={\mathcal{B}}$ has rank~$3$.  We
may take ${\mathbb{RP}}_1\hookrightarrow{\mathbb{CP}}_1$ as a typical circle in
the round $2$-sphere.  It is the unique closed orbit of
${\mathrm{SL}}(2,{\mathbb{R}})$ and the moduli space of all circles may
therefore be identified with
${\mathrm{SL}}(2,{\mathbb{C}})/{\mathrm{SL}}(2,{\mathbb{R}})$.  The fibres of
${\mathcal{B}}$ are conjugate to
${\mathfrak{sl}}(2,{\mathbb{R}})\subset{\mathfrak{sl}}(2,{\mathbb{C}})$.

\section{Loxodromes} It is well-known that conformal transformations of the
round $n$-sphere preserve the round circles.  Moreover, each such circle is
{\em homogeneous\/}, i.e.~there is a subgroup of the conformal transformations
that acts transitively on the circle.  This property can be stated
infinitesimally in terms of the Lie algebra of conformal Killing fields,
specifically that the `symmetry algebra' of any fixed circle be non-zero
(see~\cite{DZ,EZ} for details).  This is especially congenial when~$n=2$, since
we may employ complex numbers as follows.  On an affine chart
${\mathbb{C}}\hookrightarrow{\mathbb{CP}}_1$, the conformal Killing fields are 
\begin{equation}\label{killing}
X=\Re(Z),\enskip\mbox{where}\enskip Z=(az^2+bz+c)\frac\partial{\partial z}.
\end{equation}
With usual conventions $z=x+iy$, recall that
$$\frac\partial{\partial z}
=\frac12\left(\frac\partial{\partial x}-i\frac\partial{\partial y}\right)$$
so if we set 
$$a=P+iQ,\quad b=2(\lambda+iF),\quad c=2(u+iv),$$
then we find that 
\begin{equation}\label{conformal_motions}\begin{array}{l}\displaystyle X
=\overbrace{u\frac\partial{\partial x}
+v\frac\partial{\partial y}}^{\makebox[0pt]{translations}}
+\overbrace{\lambda\Big(x\frac\partial{\partial x}
+y\frac\partial{\partial y}\Big)}^{\makebox[0pt]{dilations}}
+\overbrace{F\Big(x\frac\partial{\partial y}
-y\frac\partial{\partial x}\Big)}^{\makebox[0pt]{rotations}}\\[12pt]
\displaystyle\hspace{50pt}{}
+\underbrace{P\Big(\frac{x^2-y^2}2\frac\partial{\partial x}
+xy\frac\partial{\partial y}\Big)
+Q\Big(\frac{x^2-y^2}2\frac\partial{\partial y}
-xy\frac\partial{\partial x}\Big)}_{\makebox[0pt]{inversions}},
\end{array}\end{equation}
as expected.  If we now fix $p,q\in{\mathbb{C}}$ and 
$\beta\in{\mathbb{R}}_{>0}$, then the curve
\begin{equation}\label{lox_from_p_to_q}{\mathbb{R}}\ni\theta\mapsto
z=pq\frac{e^{(\beta+i)\theta}-1}{pe^{(\beta+i)\theta}-q}\end{equation}
satisfies
\begin{equation}\label{derivative_of_theta}\frac{d\theta}{dz}
=(\beta+i)\Big[\frac1{p-q}z^2-\frac{p+q}{p-q}z+\frac{pq}{p-q}\Big]
\end{equation}
and looks like this:
\begin{center}\begin{picture}(73,73)
\put(0,0){\includegraphics[scale=.13]{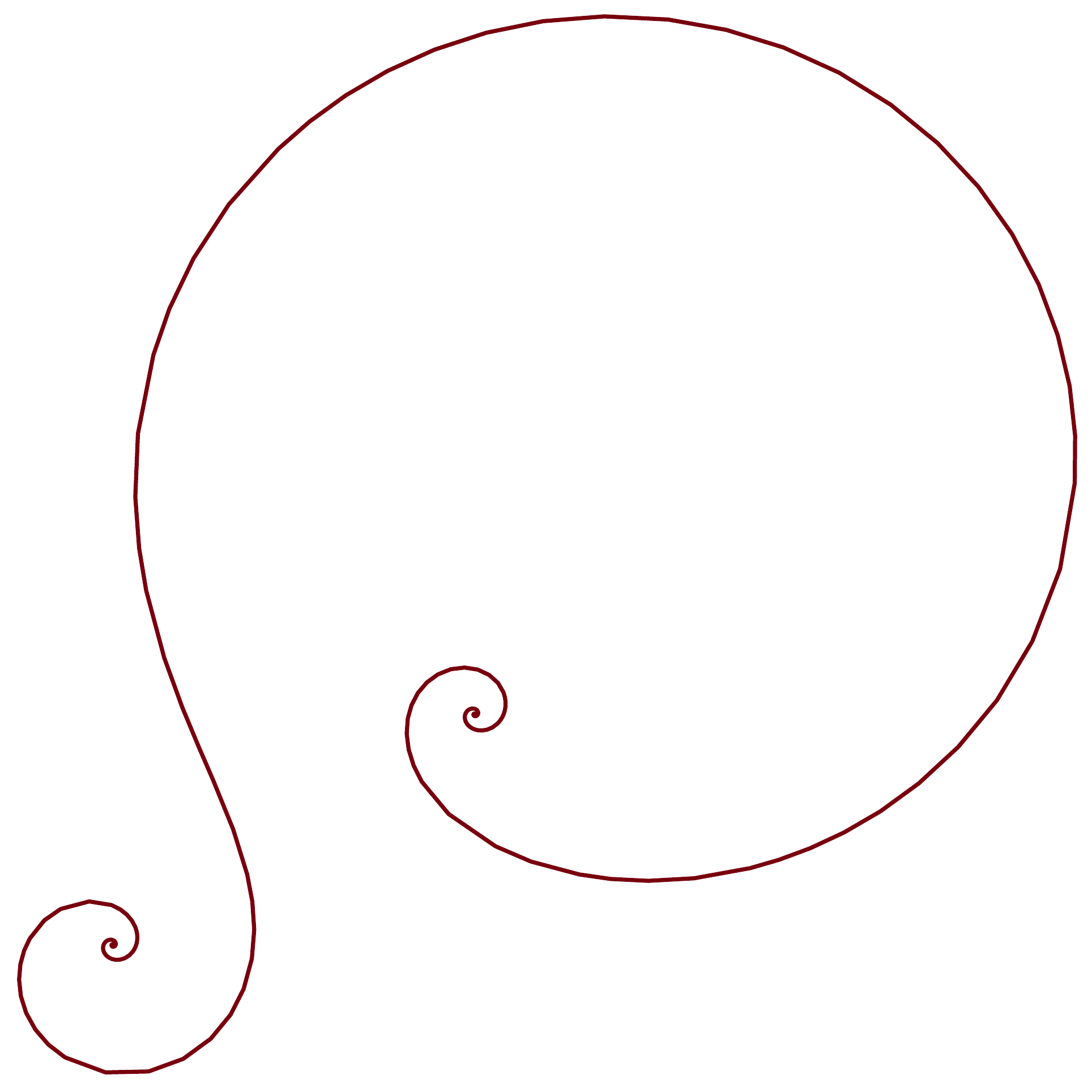}}
\put(-11,8){$p$}
\put(-4.5,8){\scriptsize$\longrightarrow$}
\put(17,23.2){$q$}
\put(23,23.2){\scriptsize$\rightarrow$}
\end{picture}\end{center}
It follows from (\ref{killing}) and (\ref{derivative_of_theta}) that 
\begin{equation}\label{generator}
\Re\left((\beta+i)\Big[\frac1{p-q}z^2-\frac{p+q}{p-q}z+\frac{pq}{p-q}\Big]
\frac\partial{\partial z}\right)\end{equation}
is a conformal symmetry of (\ref{lox_from_p_to_q}).  Alternatively, if we make 
a fractional linear change of co\"ordinates
$$\zeta=\frac{qz-pq}{pz-qp}$$
so that $p\mapsto 0$ and $q\mapsto\infty$, then the curve 
(\ref{lox_from_p_to_q}) becomes the familiar logarithmic spiral
\begin{equation}\label{log_spiral}
\zeta = e^{(\beta+i)\theta}\qquad
\raisebox{-20pt}{\begin{picture}(50,47)(0,25)
\put(0,22){\includegraphics[scale=.1]{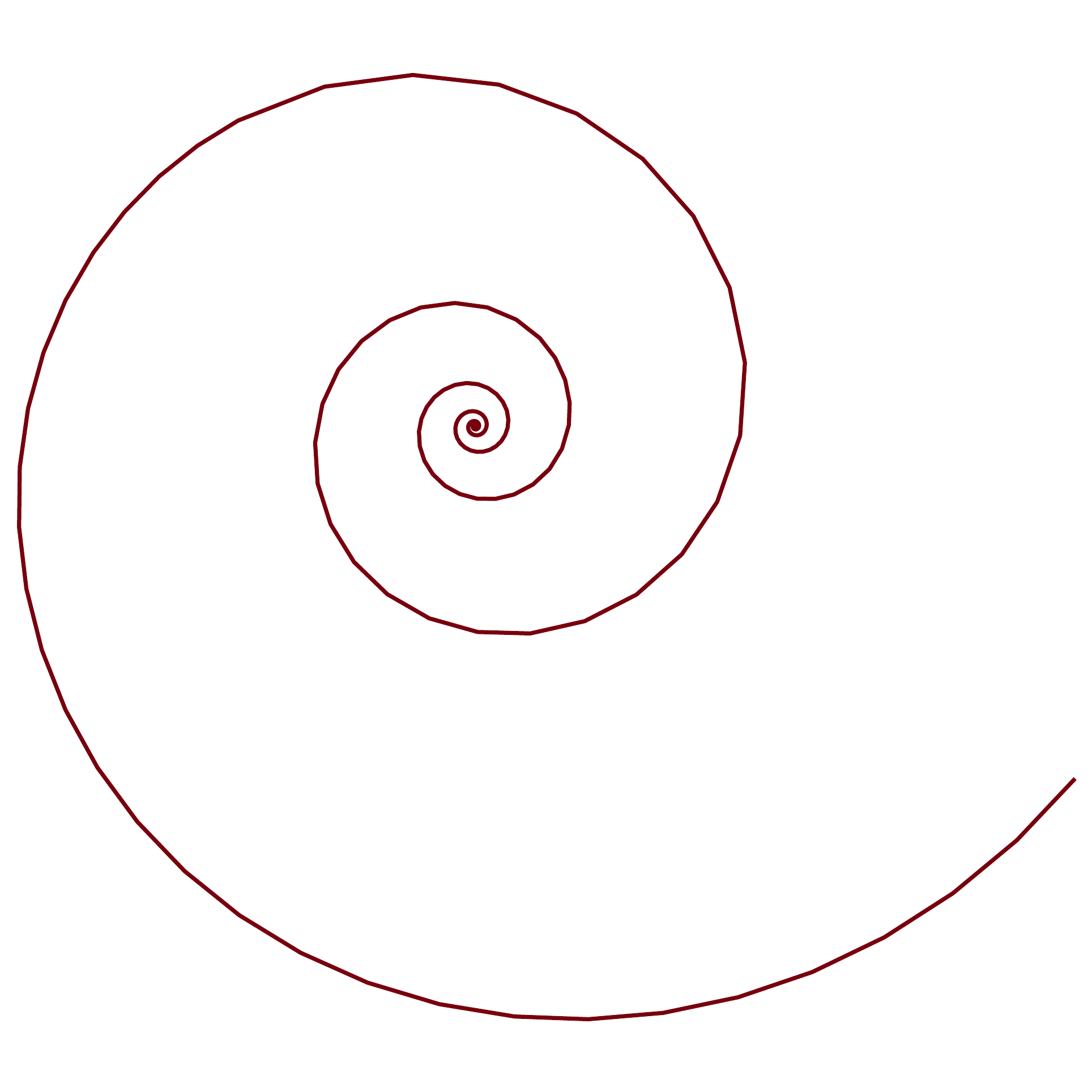}}
\end{picture}}\end{equation}
and under the Mercator projection $\zeta=e^{2\pi i(u-iv)}$ becomes the curve
$v=\beta u$ of constant bearing~$\beta$, known in navigation as a {\em Rhumb
line\/} or {\em loxodrome\/}.  Setting $\beta=0$ in (\ref{log_spiral}) gives a
circle as does letting $\beta\uparrow\infty$.  In the Mercator projection these
circles are the equator or prime meridian, respectively.  Otherwise, it is
clear that $\beta$ is a conformal invariant.  Just like household screws,
loxodromes come in two types: \begin{center}\begin{picture}(80,80)
\put(-50,-7){\includegraphics[scale=.27]{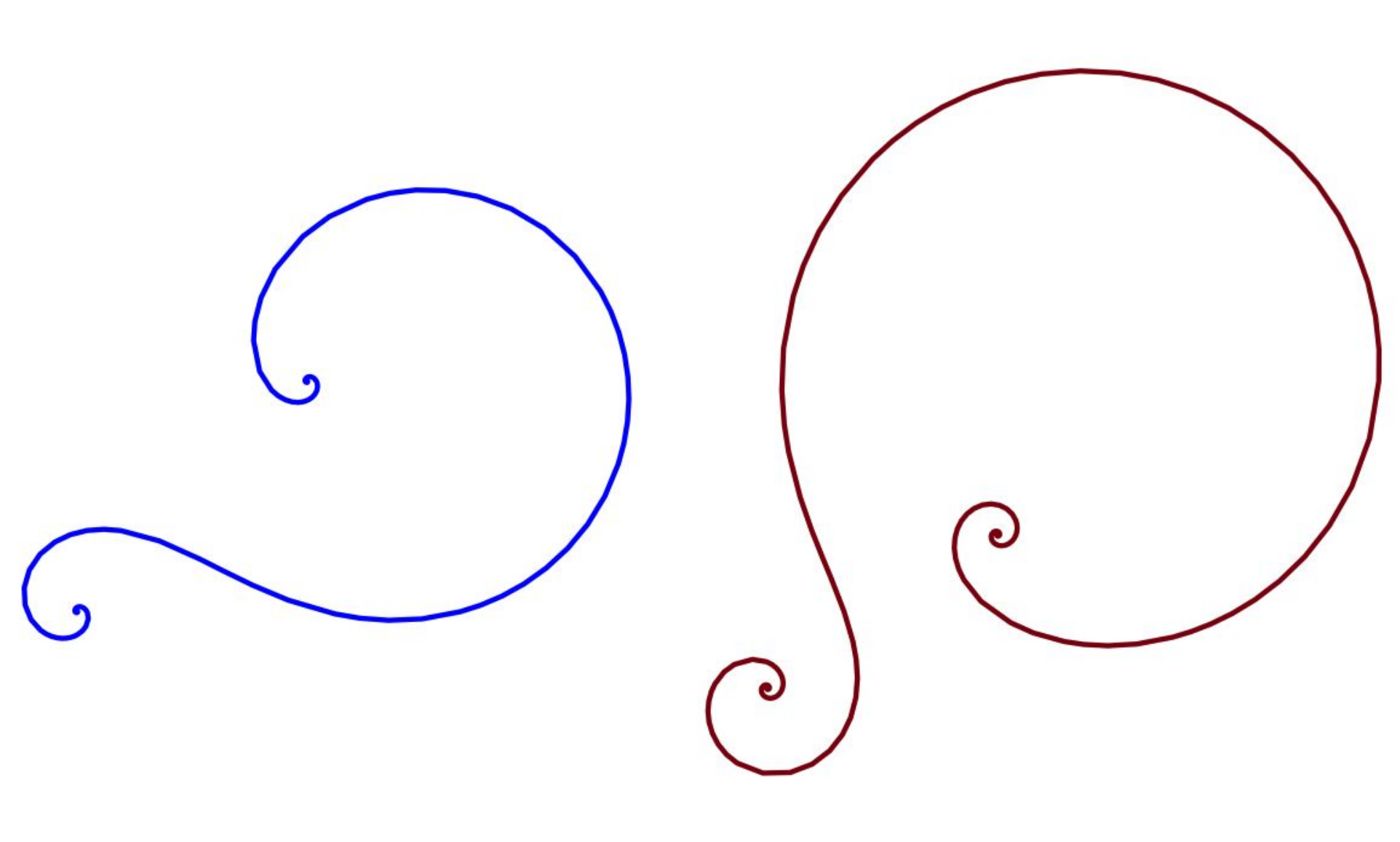}}
\put(42,40){right-handed\quad $\beta>0$} \put(-37,30){left-handed}
\put(-37,7){$\beta<0$}
\end{picture}\end{center}
and henceforth we shall concentrate on {\em right-handed loxodromes\/}, in
which case we can extract the conformal invariant $\beta$ from the discriminant
of its generator (\ref{generator}).  Specifically
$${\mathrm{discriminant}}
\left((\beta+i)\Big[\frac1{p-q}z^2-\frac{p+q}{p-q}z+\frac{pq}{p-q}\Big]\right)
=(\beta+i)^2=\beta^2-1+2\beta i,$$
which, up to real scale, yields
\begin{equation}\label{normalised_discriminant}
\frac{\beta^2-1}\beta=\beta-\frac1\beta,\end{equation}
which, in turn, precisely determines the bearing $\beta>0$.  In fact, up to
scale, there are no other conformal symmetries of the loxodrome
(\ref{lox_from_p_to_q}).  One can check this by calculation (following
\cite[\S2]{EZ}, for example) but also by counting degrees of freedom as
follows.  For fixed bearing~$\beta$, the formula (\ref{lox_from_p_to_q}) gives
a loxodrome with the origin as a marked point.  Counting $2$ dimensions for
each of $p$ and $q$ together with an additional $2$ dimensions for an arbitrary
translation therefore gives a $6$-dimensional space of marked right-handed
loxodromes with bearing~$\beta$.  Erasing the mark gives a $5$-dimensional
moduli space of right-handed loxodromes with bearing~$\beta$.  This moduli
space is homogeneous under the $6$-dimensional symmetry group
${\mathrm{SL}}(2,{\mathbb{C}})$ and so the stabiliser subgroup of any 
particular loxodrome must be just $1$-dimensional. 

\begin{theorem}[\rm Cf.~\S3.2.1 of~\cite{S}]\label{2d-classification} 
On the round two-sphere, the following
list constitutes a classification of the conformally homogeneous
unparameterised curves:
\begin{itemize}
\item the round circles,
\item the right-handed loxodromes with bearing $\beta>0$,
\item the left-handed loxodromes with bearing $\beta<0$.
\end{itemize}
\end{theorem}
\begin{proof} On the round two-sphere, conformally identified as the Riemann 
sphere ${\mathbb{CP}}_1$, the conformal Killing fields are the same as the 
global holomorphic vector fields, which in turn may be realised as the space 
of binary quadrics
$$\textstyle\Gamma({\mathbb{CP}}_1,\Theta)
\cong\Gamma({\mathbb{CP}}_1,{\mathcal{O}}(2))
\cong\bigodot^2\!{\mathbb{C}}^2,$$
as in (\ref{killing}).  Without loss of generality, we may always suppose that
such a quadric has a zero at $\infty\in{\mathbb{CP}}_1$, equivalently that
$a=0$ in the field $Z$ from~(\ref{killing}).  Then, by translation, we may
suppose that
$$Z=c\frac\partial{\partial z}\quad\mbox{or}\quad
Z=bz\frac\partial{\partial z}$$
according to whether $Z$ has a double zero at $\infty\in{\mathbb{CP}}_1$ or
not.  Looking at (\ref{conformal_motions}), we now see that, in the first case,
we encounter only translations whose flows are straight lines in the affine
chart ${\mathbb{C}}\hookrightarrow{\mathbb{CP}}_1$ and, therefore, circles
on~${\mathbb{CP}}_1$. In the second case (\ref{conformal_motions}) gives
$$\overbrace{\lambda\Big(x\frac\partial{\partial x}
+y\frac\partial{\partial y}\Big)}^{\makebox[0pt]{dilations}}
+\overbrace{F\Big(x\frac\partial{\partial y}
-y\frac\partial{\partial x}\Big)}^{\makebox[0pt]{rotations}}$$
whose flows are logarithmic spirals (\ref{log_spiral}) if both $\lambda$ and
$F$ are non-zero.  Otherwise the non-trivial flows are circles centred on the
origin, if $\lambda=0$, or rays through the origin, if $F=0$.
\end{proof}

The four main compass directions, north, south, east, and west are called the 
{\em cardinal\/} directions. The intermediate directions such as north-west are 
known as {\em ordinal\/} and it is clear from (\ref{normalised_discriminant}) 
that loxodromes with ordinal bearings are generated precisely by holomorphic 
fields $Z$ with purely imaginary discriminant. 

\section{Conformal loxodromes} Perhaps these should more accurately be called
{\em M\"obius loxodromes\/}.  They are the distinguished unparameterised curves
in $2$-dimensional M\"obius geometry that are modelled on loxodromes, just as
conformal circles are modelled on round circles.  Recall that conformal circles
are controlled by (\ref{cc-equation}), an invariant third order ODE\@.  We
shall find that conformal loxodromes are controlled by an invariant fifth order
ODE, which we may derive as follows.

In fact, for simplicity, we shall just derive the equation for {\em ordinal\/}
loxodromes.  To do this, we firstly need to write the discriminant in terms of
adjoint tractors.
\begin{lemma}\label{discriminant} In two dimensions, the discriminant of 
$$\left[\begin{array}{c}\sigma_b\\
\mu_{bc}\quad\nu\\
\rho_b\end{array}\right]\quad\mbox{is}\quad
2(4\sigma^b\rho_b-\mu^{bc}\mu_{bc}+2\nu^2)
+4i\epsilon^{bc}(\nu\mu_{bc}-2\sigma_{[b}\rho_{c]}).$$
\end{lemma}
\begin{proof}
In flat space with co\"ordinates $(x^1,x^2)=(x,y)$, the general conformal
Killing field (\ref{conformal_motions}) corresponds to the following section of
the adjoint tractor bundle
$$\begin{array}{c}
\sigma_1=u+\lambda x-Fy+\frac12P(x^2-y^2)-Qxy\\[2pt]
\sigma_2=v+\lambda y+Fx+Pxy+\frac12Q(x^2-y^2)\\[4pt]
\mu_{12}=F+Py+Qx\quad\qquad\nu=\lambda+Px-Qy\\[4pt]
\rho_1=-P\enskip\quad\rho_2=Q\,,
\end{array}$$
which is easily checked to be parallel for the adjoint tractor connection.  (In
other words, this section of ${\mathcal{A}}$ is the result of applying the
usual BGG splitting operator~\cite{HSSS} to the conformal Killing field $X$ in
(\ref{conformal_motions}).)  For this section, a computation yields
$$\begin{array}{rcl}
2(4\sigma^b\rho_b-\mu^{bc}\mu_{bc}+2\nu^2)&=&4(\lambda^2-F^2-2Pu+2Qv)\\[4pt]
4\epsilon^{bc}(2\sigma_{[b}\rho_{c]}-\nu\mu_{bc})
&=&8(Qu+Pv-\lambda F),
\end{array}$$
where we have chosen to normalise $\epsilon^{bc}$ such that $\epsilon^{12}=1$.
But, with the conventions leading up to~(\ref{conformal_motions}), it is easily
checked that
$$b^2-4ac=4\big[(\lambda^2-F^2-2Pu+2Qv)+2i(\lambda F-Qu-Pv)\big],$$
as advertised.
\end{proof}
In arbitrary dimensions, it is interesting to note that
$$4\sigma^b\rho_b-\mu^{bc}\mu_{bc}+2\nu^2\quad\mbox{and}\quad
2\sigma_{[b}\rho_{c]}-\nu\mu_{bc}$$
are invariants of adjoint tractors (with values in $\Wedge^0$ and
$\Wedge^2[2]$, respectively).  In two dimensions, Lemma~\ref{discriminant}
implies that ordinal loxodromes in flat space are generated by parallel adjoint
tractors with $4\sigma^b\rho_b-\mu^{bc}\mu_{bc}+2\nu^2=0$.  The tractor
interpretation of the Doubrov-\v{Z}\'adn\'{\i}k formulation~\cite{DZ} of
unparameterised distinguished curves modelled on ordinal loxodromes in M\"obius
geometry hence requires that we find a section of ${\mathcal{A}}$ of the form
$$\left[\begin{array}{c}\sigma_b\\
\mu_{bc}\quad\nu\\
\rho_b\end{array}\right]
=\left[\begin{array}{c}U_b\\
\rule{10pt}{2pt}\quad\rule{10pt}{2pt}\\
\rule{10pt}{2pt}\end{array}\right]\quad\mbox{such that}\enskip
4\sigma^b\rho_b-\mu^{bc}\mu_{bc}+2\nu^2=0$$
along $\gamma\hookrightarrow M$ and so that this section is preserved up to
scale by $\partial$ along~$\gamma$.  It is straightforward though tedious to
figure out what $\mu_{bc}$, $\nu$, and $\rho_b$ are forced to be: 
\begin{lemma}\label{this_is_forced}
For any smooth oriented curve $\gamma\hookrightarrow M$, in the presence of a
background metric, wherever $J_b\not=0$ let us define a density
$\kappa$ of conformal weight $-1$ along $\gamma$ by
\begin{equation}\label{def_of_kappa}
\partial J_b+A^cJ_cU_b+2\kappa J_b=0.\end{equation}
Then the formula
\begin{equation}\label{lift_of_U}\left[\begin{array}{c}\sigma_b\\
\mu_{bc}\quad\nu\\
\rho_b\end{array}\right]
=\left[\begin{array}{c}U_b\\
2U_{[b}A_{c]}\quad\kappa\\
J_b+\frac12(A^cA_c-\kappa^2)U_b+\kappa A_b\end{array}\right]\end{equation}
gives an invariantly defined section of ${\mathcal{A}}[-1]$ along~$\gamma$, 
having 
\begin{equation}\label{null}
4\sigma^b\rho_b-\mu^{bc}\mu_{bc}+2\nu^2=0.\end{equation}
\end{lemma}
\begin{proof}
As observed at the end of the proof of Theorem~\ref{mini_thm}, we have 
$U^bJ_b=0$ along $\gamma$ so $U_b,J_b$ are linearly independent wherever 
$J_b\not=0$. It follows that $\kappa$ is well-defined by~(\ref{def_of_kappa}) 
and from the invariance (\ref{jerk-transform}) of $J_b$ and the formula 
(\ref{A-transform}) for $\widehat A_a$ it follows that
\begin{equation}\label{kappa-transform}
\widehat\kappa=\kappa+U^a\Upsilon_a.\end{equation}
This is part of what needs to be shown (\ref{adjoint_transformation}) in order
to verify that (\ref{lift_of_U}) be invariantly defined.  The remaining
verifications are straightforward though tedious.  Finally, a simple
calculation checks that (\ref{null}) holds.
\end{proof}
In fact, any coefficient in front of $J_b$ on the last line of
(\ref{lift_of_U}) will give an equally good invariant lift of~$U_b$ to
${\mathcal{A}}[-1]$ but the following simple formula depends crucially on that
coefficient being~$1$.  
\begin{lemma}\label{apply_partial} For the canonical section \eqref{lift_of_U}
of ${\mathcal{A}}[-1]$, in the presence of a background metric, we have
$$\textstyle\partial\left[\begin{array}{c}\sigma_b\\
\mu_{bc}\quad\nu\\
\rho_b\end{array}\right]
=(\partial\kappa+\frac12(A^aA_a+\kappa^2)+U^aU^c\Rho_{ac})
\left[\begin{array}{c}0\\
0\quad 1\\
A_b-\kappa U_b\end{array}\right]
-\kappa\left[\begin{array}{c}\sigma_b\\
\mu_{bc}\quad\nu\\
\rho_b\end{array}\right].$$
\end{lemma}
\begin{proof} A straightforward computation.
\end{proof}
Combining Lemmata~\ref{this_is_forced} and~\ref{apply_partial} immediately
leads to the following conclusion.  
\begin{theorem} Suppose $M$ is a smooth $2$-dimensional manifold equipped with
a M\"obius structure and suppose that $\gamma\hookrightarrow M$ is a smooth
oriented curve with non-vanishing normalised jerk~$J_b$.  Then the invariantly
defined line subbundle
$${\mathrm{span}}\left(\left[\begin{array}{c}U_b\\
2U_{[b}A_{c]}\quad\kappa\\
J_b+\frac12(A^cA_c-\kappa^2)U_b+\kappa A_b\end{array}\right]\right),$$
where $\partial J_b+A^cJ_cU_b+2\kappa J_b=0$ defines~$\kappa$, is preserved by
$\partial$ along $\gamma$ if and only if
\begin{equation}\label{ordinal_loxodrome}
\textstyle\partial\kappa+\frac12(A^aA_a+\kappa^2)+U^aU^b\Rho_{ab}=0.
\end{equation}
\end{theorem}
We may formalise the conclusion of this theorem as follows.
\begin{definition}An {\em ordinal conformal loxodrome\/} is a smooth curve 
$\gamma\hookrightarrow M$ along which $J_b\not=0$ and (\ref{ordinal_loxodrome})
holds, where $\kappa$ is defined by~(\ref{def_of_kappa}).
\end{definition}

We remark that (\ref{ordinal_loxodrome}) is M\"obius invariant by construction 
but, of course, this can be verified directly as follows:
\begin{itemize}
\item the density $\kappa$ has weight $-1$ 
and $\widehat\kappa=\kappa+U^a\Upsilon_a$ so
$$\begin{array}{rcl}\widehat\partial\widehat\kappa
&=&\partial(\kappa+U^a\Upsilon_a)-U^b\Upsilon_b(\kappa+U^a\Upsilon_a)\\[4pt]
&=&\partial\kappa+A^a\Upsilon_a+U^a\partial\Upsilon_a
-U^a\Upsilon_a\kappa-(U^a\Upsilon_a)^2,\end{array}$$
\item from (\ref{A-transform}) we find
$$\widehat A^a\widehat A_a
=A^aA_a+\Upsilon^a\Upsilon_a-(U^a\Upsilon_a)^2-2A^a\Upsilon_a,$$
\item from (\ref{kappa-transform}) we find
$$\widehat\kappa^2=\kappa^2+2U^a\Upsilon_a\kappa+(U^a\Upsilon_a)^2,$$
\item from (\ref{Rho-transform}) we find
$$\textstyle U^aU^b\widehat\Rho_{ab}
=U^aU^b\Rho_{ab}-U^a\partial\Upsilon_a
+(U^a\Upsilon_a)^2-\frac12\Upsilon^a\Upsilon_a,$$
\end{itemize}
and, since all correction terms cancel, it follows that
$\partial\kappa+\frac12(A^aA_a+\kappa^2)+U^aU^b\Rho_{ab}$ is an invariant
density of weight~$-2$.  Of course, this construction is all under the
assumption that $J_a\not=0$, which means that we are avoiding curves that, at
some point, osculate a conformal circle to third order.

\section{A fourth order ODE} 

At this point, one may smell something fishy
since, in~\cite{DK}, Dunajski and Kry\'nski manufacture a {\em fourth\/} order
conformally invariant ODE \cite[Equation~(4.14)]{DK} that, in flat two
dimensions, simultaneously controls the conformal circles and loxodromes
whereas (\ref{ordinal_loxodrome}) is {\em fifth\/} order and controls just the
ordinal loxodromes.  This smell comes from the red herring mentioned early
in~\S\ref{cc} regarding whether we are considering {\em parameterised\/} or
{\em unparameterised\/} curves.  In~\cite{DK}, curves $\gamma\hookrightarrow M$
are equipped with a parameterisation $t:\gamma\to{\mathbb{R}}$, which one may
use to normalise various quantities as follows.  Fixing attention on a
particular curve~$\gamma$, we may always pick a metric $g_{ab}$ in the
conformal class so that the given parameterisation $t$ is by arc-length.
Having done this, the remaining freedom 
$g_{ab}\mapsto\widehat g_{ab}=\Omega^2g_{ab}$ is constrained by requiring that
$\Omega|_\gamma=1$ and it follows that $\partial\Omega=0$ along~$\gamma$, in
other words that
\begin{equation}\label{normalised_by_parameterisation}
U^a\Upsilon_a=0,\quad\mbox{where}\enskip
\Upsilon_a=\Omega^{-1}\nabla_a\Omega.\end{equation}
Recall that the normalised jerk $J_b$, defined by (\ref{jerk}), is conformally
invariant of weight~$-1$. It follows that
$$\widehat\partial J_b=\partial J_b-2U^a\Upsilon_aJ_b-U^aJ_a\Upsilon_b
+\Upsilon^aJ_aU_b.$$
But recall that $U^aJ_a=0$, so from (\ref{A-transform}) we conclude that the
{\em normalised snap\/}
$$S_a\equiv\partial J_a+A^bJ_bU_a$$
transforms quite simply by
$$\widehat S_a=S_a-2U^b\Upsilon_bJ_a.$$
Therefore, along a {\em parameterised curve\/} with conformal factor
constrained by~(\ref{normalised_by_parameterisation}), we find that the
normalised snap is conformally (or M\"obius) invariant.  Therefore, the 
equation $S_a=0$ is an invariant fourth order ODE on {\em parameterised\/} 
curves, which is certainly satisfied along conformal circles.

To compare with~\cite{DK}, let us introduce the {\em Cotton-York tensor\/}
$Y_{abc}\equiv 2\nabla_{[a}\Rho_{b]c}$, which is M\"obius invariant in two
dimensions~\cite{C}, conformally invariant in three dimensions~\cite{BEG}, and 
in higher dimensions transforms by 
\begin{equation}\label{Y-transform}
\widehat Y_{abc}=Y_{abc}-W_{abc}{}^d\Upsilon_d,\end{equation} 
where $W_{abcd}$, the {\em conformal Weyl tensor\/}, is the trace-free part of
Riemann tensor and is conformally invariant.  The {\em contracted Bianchi
identity\/}
$$\nabla^dW_{abcd}+(n-3)Y_{abc}=0$$
for $n\geq 4$ also shows $Y_{abc}=0$ in the conformally flat case.  Indeed,
tractor calculus \cite{BEG,C} easily shows that $Y_{abc}$ is the obstruction to
M\"obius or conformal flatness in dimensions $2$ or $3$, whilst $W_{abcd}$ is
the obstruction in all higher dimensions. 

{From} (\ref{A-transform}) we see that 
$$W_{abc}{}^dU^c\widehat A_d=W_{abc}{}^dU^cA_d-W_{abc}{}^dU^c\Upsilon_d$$
and so from (\ref{Y-transform}) it follows that the $2$-form 
$$K_{ab}\equiv W_{abc}{}^dU^cA_d-Y_{abc}U^c$$
is M\"obius invariant in dimension $2$ and conformally invariant in all
dimensions $n\geq 3$ (the Weyl tensor vanishes by symmetry considerations in
dimensions $2$ and $3$).  The $2$-form $K_{ab}$ evidently vanishes in the
M\"obius or conformally flat cases.  (An invariant r\^{o}le for $K_{ab}$
remains mysterious but it certainly looks like $K_{ab}$ should arise from some
tractor curvature~\cite[\S2.5]{BEG}.)

Josef \v{S}ilhan has checked that, with the notation introduced above, the
invariant fourth order ODE \cite[Equation (4.14)]{DK} reads $S_a=K_{ab}U^b$.
When $K_{ab}=0$, solutions of this equation evidently include the conformal
circles $J_a=0$ but otherwise, as unparameterised curves in flat
two-dimensional space, must be loxodromes in accordance with
Theorem~\ref{2d-classification} above and this is confirmed by calculation
in~\cite[\S5]{DK}.

I would like to thank Maciej Dunajski for drawing my attention to~\cite{DK} 
after I had carelessly forgotten its contents.

\section{Outlook} The novelty of conformal loxodromes in $2$ dimensions stems
from their conformal invariance on the flat model, the round $2$-sphere.  What
I believed to be missing for conformal geometry in higher dimensions was a
classification of homogeneous curves on the round $n$-sphere, generalising
Theorem~\ref{2d-classification} above.  However, it was pointed out to me by an
anonymous referee that such a classification has been carried out for the
$3$-sphere in an unpublished work of Sulanke~\cite[Theorem~1]{S}.  Not only
that, but his proof identifies these homogeneous curves as orbits of certain
$1$-parameter subgroups of the group of conformal motions leading directly to
their symmetry algebras via an algorithm of~\cite{DKR} (reproduced in this
context as~\cite[Lemma~6]{EZ}).  Thus, we may insist, in higher dimensions,
that $\partial J_b$ lie in ${\mathrm{span}}(U_b,J_b)$ so that $\kappa$ may be
defined by~(\ref{def_of_kappa}) and then proceed to define a class of invariant
curves by~(\ref{ordinal_loxodrome}).  We could also call these curves
`loxodromes' but there are yet more homogeneous curves in the $n$-sphere for
$n\geq 3$ and hence yet more analogues in the curved setting.  It would surely
be useful to classify these invariant curves in all dimensions and identify the
ODE that control them.  Conformal circles are proving themselves to be
increasingly useful in conformal differential geometry where important results,
for example~\cite{CDT}, rely on knowing explicitly the ODE that controls where
they go.  Conformal loxodromes and the multitude of other conformally invariant
curves may prove similarly useful.  Of course, there are corresponding
questions in all parabolic geometries.

The anonymous referee also points out a general argument in~\cite[\S3.2]{S},
which says that the homogeneous curves in a homogeneous space $G/H$ are exactly
those having `constant curvatures' and that, for example, the requirement that
$\partial J_b$ lie in ${\mathrm{span}}(U_b,J_b)$ may be interpreted as the
curves with vanishing `conformal torsion.'  The referee suggests, therefore,
that the corresponding curves in the curved setting may also be approached as
restrictions on their conformal curvatures.  There are nice classical
formul{\ae} for these curvatures (in three dimensions see~\cite{CSW}).  It
would be interesting to compare the resulting theory with the ODE, yet to be
derived via the method of Doubrov-\v{Z}\'adn\'{\i}k~\cite{DZ}, as were obtained
for loxodromes in this article.

\end{document}